\documentclass{amsart}
\usepackage{amsmath,amsfonts,amssymb,amscd,enumerate,mathtools, color, comment}

\newtheorem{theorem}{Theorem}[section]

\newtheorem{lemma}[theorem]{Lemma}

\numberwithin{equation}{section}

\title{The Ces\`{a}ro operator on local Dirichlet spaces}

\author{E. Dellepiane}
\address{%
Dipartimento di Matematica ``Federigo Enriques", Università degli studi di Milano, Milano, Italy 20133}
\email{eugenio.dellepiane@unimi.it}

\author{J. Mashreghi}
\address{%
D\'{e}partement de math\'{e}matiques et statistique,
Universit\'{e} Laval,
Qu\'{e}bec, QC,
Canada G1K 7P4}
\email{javad.mashreghi@mat.ulaval.ca}

\author{M. Nasri}
\address{Department of Mathematics and Statistics,
University of  Winnipeg,
Winnipeg, MB,
Canada R3B 2E9}
\email{m.nasri@uwinnipeg.ca}

\author{W. Verreault}
\address{Department of Mathematics,
University of Toronto,
Toronto, ON,
Canada M5S 2E4}
\email{william.verreault@utoronto.ca}



\thanks{This work was supported by grants NSERC and the Canada Research Chairs program. The first author is a member of Gruppo Nazionale per l'Analisi Matematica, la Probabilità e le loro Applicazioni (GNAMPA) of Istituto Nazionale di Alta Matematica (INdAM)}

\begin{document}

\begin{abstract}
The family of Ces\`{a}ro operators $\sigma_n^\alpha$, $n \geq 0$ and $\alpha \in [0,1]$, consists of finite rank operators on Banach spaces of analytic functions on the open unit disc. In this work, we investigate these operators as they act on the local Dirichlet spaces $\mathcal{D}_\zeta$. It is well-established that they provide a linear approximation scheme when $\alpha > \frac{1}{2}$, with the threshold value $\alpha = \frac{1}{2}$ being optimal. We strengthen this result by deriving precise asymptotic values for the norm of these operators when $\alpha \leq \frac{1}{2}$, corresponding to the breakdown of approximation schemes. Additionally, we establish upper and lower estimates for the norm when $\alpha > \frac{1}{2}$.
\end{abstract}

\maketitle

\section{Introduction}

The classical Dirichlet space $\mathcal{D}$ \cite{MR3969961, MR3185375}, the Hardy spaces $H^p$, $0<p\leq \infty$, \cite{MR0268655, MR1669574, MR2500010}, and the Bergman space $A^2$ \cite{MR0507701, MR2033762, MR1758653} are the most well-known function spaces on the open unit disc $\mathbb{D}$. These spaces have been extended in several diverse directions. Our focus in this work is on local Dirichlet spaces $\mathcal{D}_\zeta$, for which we study the growth of generalized Ces\`{a}ro operators. For a brief history of Dirichlet spaces and a comprehensive study of local Dirichlet spaces, we refer to the monograph \cite{MR3185375}. Some definitions and elementary properties of Dirichlet spaces  needed in our discussion are gathered in Section \ref{S:wds}.

One of the fundamental facts in complex analysis is that every holomorphic function $f$ on the open unit disc $\mathbb{D}$ has the Taylor series expansion
\begin{equation}\label{E:intro-0}
f(z)= \sum_{k=0}^{\infty} a_{k} z^{k}, \qquad z \in \mathbb{D}.
\end{equation}
Therefore, in any polynomial approximation scheme in a Banach space $\mathcal{X}$ consisting of such analytic functions, it is natural to consider the Taylor polynomials
\begin{equation}\label{E:intro-4}
\left(S_{n} f\right)(z):=\sum_{k=0}^{n} a_{k} z^{k}, \qquad n \geq 0,
\end{equation}
and explore if $S_nf \to f$ in the ambient space $\mathcal{X}$. For example, it is straightforward to see that this holds in the Hardy--Hilbert space $H^2$, and in the classical Dirichlet space $\mathcal{D}$. Moreover, it is a deep result of Hardy--Littlewood \cite{HL1, MR6581} that it is also a valid approximation scheme in $H^p$-spaces for $1<p<\infty$. However, this natural approximation method fails in some settings such as the disc algebra $\mathcal{A}$ and the Hardy space $H^1$.

As a first alternative to Taylor polynomials \eqref{E:intro-4}, the weighted versions
\begin{equation}\label{E:intro-4-cesaro}
\left(\sigma_{n} f\right)(z) := \sum_{k=0}^{n} \left( 1-\frac{k}{n+1} \right) a_{k} z^{k}, \qquad n \geq 0,
\end{equation}
known as Ces\`{a}ro means or Fej\'{e}r polynomials, were considered. In fact, Hardy--Littlewood demonstrated that $\sigma_{n} f \to f$ holds in both the disc algebra $\mathcal{A}$ and the Hardy space $H^1$. They even considered more sophisticated means which we do not discuss in this work. Several decades later, recognizing that $S_nf$ may not be suitable for super-harmonically weighted Dirichlet spaces $\mathcal{D}_w$, Mashreghi--Ransford \cite{MR4032210} established a similar result for $\mathcal{D}_w$. More explicitly, it was demonstrated that while there are cases where $\|S_nf-f\|_{\mathcal{D}_w} \not\to 0$, the approximation scheme $\|\sigma_nf-f\|_{\mathcal{D}_w} \to 0$ remains valid.

In \cite{MR4179962}, the generalized Ces\`{a}ro means \[
(\sigma_n^{\alpha} f)(z)=\binom{n+\alpha}{\alpha}^{-1}\sum_{k=0}^n\binom{n-k+\alpha}{\alpha}a_kz^k
\]
were considered on super-harmonically weighted Dirichlet spaces. Here, $\alpha$ is a parameter in the interval $[0,1]$, and notice that for $\alpha=0$ and $\alpha=1$ one recovers, respectively, $S_n$ and $\sigma_n$. It was shown that the approximation $\|\sigma_n^\alpha f-f\|_{\mathcal{D}_w} \to 0$ is valid for every weight $\omega$ if and only if $\alpha>\frac{1}{2}$. In this work, we are interested in the asymptotic behavior of the norm of $\sigma_n^\alpha$ on super-harmonically weighted Dirichlet spaces, as $n\to\infty$. The case $\alpha \leq \frac{1}{2}$, where the approximation fails, is particularly more interesting.

The structure of the paper is as follows. In Section \ref{S:wds}, we provide a brief overview of some technical results necessary for presenting our main findings. We commence with the definition of weighted Dirichlet spaces, with particular emphasis on local Dirichlet spaces. Following that, we introduce the concept of Hadamard products, complemented by Theorem \ref{T:norm-hadamard-h} from \cite{MR4032210}, which plays a significant role in our studies. This section concludes with a discussion on generalized Cesàro means $\sigma_{n}^{\alpha}$. For more detailed information on these operators, we refer to \cite{MR1188874}. Section \ref{S:main-results} presents our main findings, consisting of three theorems: Theorems \ref{T:alpha<12}, \ref{T:alpha=12}, and \ref{T:alpha>12}. Given the calculatory nature and length of the proofs, Section \ref{S:estimation-Lemma} compiles several estimation results that are noteworthy in their own right. However, these estimations also play a crucial role in proving the main results, which are detailed in Section \ref{S:proofs}. Finally, the paper concludes with a brief section containing some concluding remarks.

\section{Technical Results} \label{S:wds}

\subsection{Local Dirichlet Spaces}
The classical \textit{Dirichlet space} $\mathcal{D}$ consists of functions $f\in \operatorname{Hol}(\mathbb{D})$  for which
\[
\mathcal{D}(f):=\frac{1}{\pi} \int_{\mathbb{D}}\left|f^{\prime}(z)\right|^{2} d A(z) < \infty.
\]
Here, $dA$ denotes the  Lebesgue area measure in $\mathbb{D}$.  S. Richter \cite{MR936999} introduced the harmonically weighted and A. Aleman \cite{AL} the superharmonically weighted Dirichlet spaces. A special but very important role is played by local Dirichlet spaces consisting of functions  $f \in \operatorname{Hol}(\mathbb{D})$ with $\mathcal{D}_{\zeta}(f)<\infty$, where
\[
\mathcal{D}_{\zeta}(f)=
\begin{cases}
\displaystyle\int_{\mathbb{D}} \log \left|\frac{1-\bar{\zeta} z}{\zeta-z}\right| \frac{2}{1-|\zeta|^{2}}\left|f^{\prime}(z)\right|^{2} d A(z), & \zeta \in \mathbb{D}, \\
\\
\displaystyle\int_{\mathbb{D}} \frac{1-|z|^{2}}{|\zeta-z|^{2}}\left|f^{\prime}(z)\right|^{2} d A(z), & \zeta \in \mathbb{T}.
\end{cases}
\]
In fact, by a very influential result of Richter--Sundberg \cite{MR1116495}, an  $f \in \mbox{Hol}(\mathbb{D})$ is in $\mathcal{D}_{\zeta}$ if and only if there is a $g \in H^{2}$ and $a\in\mathbb{C}$ such that
\begin{equation}\label{DirichleHardy}
f(z)=a+(z-\zeta) g(z), \qquad z \in \mathbb{D},
\end{equation}
and, moreover, we have the more practical formula
\begin{equation}\label{DirichleHardy-2}
\mathcal{D}_{\zeta}(f)=\|g\|_{H^{2}}^{2}
\end{equation}
to calculate the local Dirichlet integral. See \cite[Chapter IV]{MR3185375}.

\subsection{Hadamard Multipliers}
The \textit{Hadamard product} of two formal power series $f(z):=\sum_{k=0}^{\infty} a_{k} z^{k}$ and $g(z):=\sum_{k=0}^{\infty} b_{k} z^{k}$ is the formal power series given by the formula $(f * g)(z):=\sum_{k=0}^{\infty} (a_{k} b_{k}) z^{k}$. It is trivial that if either $f$ or $g$ is a polynomial, then $f * g$ is also a polynomial. For a sequence of complex numbers $(c_{k})_{k\in\mathbb{N}}$, write $T_{c}$ for the infinite matrix
\begin{equation}\label{E:Definition of $T_c$}
 T_{c}:=\left(\begin{array}{ccccc}
c_{1} & c_{2}-c_{1} & c_{3}-c_{2} & c_{4}-c_{3} & \ldots \\
0 & c_{2} & c_{3}-c_{2} & c_{4}-c_{3} & \cdots \\
0 & 0 & c_{3} & c_{4}-c_{3} & \ldots \\
0 & 0 & 0 & c_{4} & \ldots \\
\vdots & \vdots & \vdots & \vdots & \ddots
\end{array}\right).
\end{equation}
If this matrix is a bounded operator on $\ell^{2}$, we denote its operator norm by $\left\|T_{c}\right\|_{\ell^2\to\ell^2}$. For convenience of notation, if  $(c_{k})_{k\in\mathbb{N}}$ is the sequence of the Taylor coefficients of an analytic function $h$, or even of a formal power series $h(z)=\sum_{k=0}^{\infty} c_{k} z^{k}$, we also write $T_{h}$ in place of $T_{c}$. Note that the term $c_0$ does not appear in the matrix $T_c$: this is consistent with the fact that the Dirichlet integrals annihilate the constants.

Here is a simple observation about eigenvalues of $T_c$. As a matter of fact, we will just use a special case of \eqref{eq-P:eigenvalues}, corresponding to $k=1$, which also follows from more substantial results in \cite{MR4032210}.

\begin{lemma} \label{P:eigenvalues}
Let $(c_{k})_{k\geq 1}$ be a sequence of complex numbers. Then $c_k$, $k \geq 1$, is an eigenvalue of the matrix $T_c$, with corresponding eigenvector
\[
\mathbf{v}_k =\sum_{i=1}^{k} \mathbf{e}_i,
\]
where $\mathbf{e}_i$ is the sequence given by $\mathbf{e}_i(j) = \delta_{ij}$, the Kronecker delta, for $j\in\mathbb{N}$. In particular, if $T_c$ is bounded on $\ell^2$,
\begin{equation}\label{eq-P:eigenvalues}
|c_k| \leq \|T_c\|_{\ell^2\to\ell^2}, \qquad k \geq 1.
\end{equation}
\end{lemma}

\begin{proof}
We prove by induction that $T_c \mathbf{v}_k = c_k \mathbf{v}_k$. For $k=1$, the result is trivial. Assume that it holds for $k\in\mathbb{N}$. Then
\begin{align*}
T_c \mathbf{v}_{k+1} &= T_c \mathbf{v}_k + T_c \mathbf{e}_{k+1} = c_k \mathbf{v}_k + \sum_{i=1}^k (c_{k+1}-c_{k})\mathbf{e}_{i} + c_{k+1} \mathbf{e}_{k+1} \\
&= c_k \mathbf{v}_k + (c_{k+1}-c_{k}) \mathbf{v}_k +c_{k+1}\mathbf{e}_{k+1} = c_{k+1}\mathbf{v}_{k+1}. \qedhere
\end{align*}
\end{proof}

The following central result is needed in our discussion.

\begin{theorem}[\cite{MR4032210}] \label{T:norm-hadamard-h}
The function $h$ is a Hadamard multiplier of $\mathcal{D}_{\omega}$, for all superharmonic weight $\omega$, if and only if $T_{h}$ acts as a bounded operator on $\ell^{2}$. Moreover, in this case, the estimate
\begin{equation} \label{MasRanThm11}
\mathcal{D}_{\omega}(h * f) \leq\left\|T_{h}\right\|_{\ell^2\to\ell^2}^{2} \mathcal{D}_{\omega}(f)
\end{equation}
holds, with the constant $\left\|T_{h}\right\|_{\ell^2\to\ell^2}^{2}$ being optimal.
\end{theorem}

To be more precise, we introduce the quantity
\begin{equation}\label{E:def-TdwtoDw}
\|T_h\|^2_{\mathcal{D}_\omega \to \mathcal{D}_\omega} := \sup_f \frac{\mathcal{D}_\omega(h\ast f)}{\mathcal{D}_\omega(f)},
\end{equation}
where $\omega$ is a superharmonic weight and the supremum is taken over all non-constant $f\in \mathcal{D}_\omega$. Hence, we know that, for each weight $\omega$,
\[
\|T_h\|_{\mathcal{D}_\omega \to \mathcal{D}_\omega} \leq \|T_h\|_{\ell^2 \to \ell^2},
\]
and the sharpness of the constant $\|T_h\|_{\ell^2\to\ell^2}$ means that when we take another supremum with respect to all weights $\omega$, we obtain
\[
\sup_{\omega} \|T_h\|_{\mathcal{D}_\omega \to \mathcal{D}_\omega} = \|T_h\|_{\ell^2 \to \ell^2}.
\]
In \cite{MR4032210}, it is shown that this supremum is attained by choosing the harmonic weight
\[
\omega_1(z) := \frac{1-|z|^2}{|1-z|^2}, \qquad z\in\mathbb{D}.
\]
The corresponding Dirichlet space is the local Dirichlet space $\mathcal{D}_1$. Our goal is to further analyze
\[\|T_h\|:=\|T_h\|_{\ell^2\to\ell^2} = \sup_f \frac{\mathcal{D}_1(h\ast f)}{\mathcal{D}_1(f)}\]
for the special class of polynomials $h$ that give rise to the generalized Ces\`{a}ro means. Proper estimation of $\left\|T_{h}\right\|$ is crucial in applications.

\subsection{The Generalized Ces\`{a}ro Means} \label{S:ces}
For the rest of this work, $n>1$ will be a fixed natural number and $\alpha$ a real number in the interval $[0,1]$. For a power series $f(z)=\sum_{k=0}^\infty a_kz^k$, the generalized Ces\`{a}ro operator $\sigma_n^\alpha$ acts on $f$ as
\[
(\sigma_n^{\alpha} f)(z)=\binom{n+\alpha}{\alpha}^{-1}\sum_{k=0}^n\binom{n-k+\alpha}{\alpha}a_kz^k,
\]
where the generalized binomial coefficient is defined for a pair of real numbers $x,y$ with $x>y>-1$ by
\[
\binom{x}{y}:=\frac{\Gamma(x+1)}{\Gamma(y+1)\Gamma(x-y+1)}.
\]
Here, $\Gamma$ denotes the Gamma function. For simplicity, we set
\begin{equation} \label{E:defck}
c_k = c_{k,n}^\alpha = \binom{n+\alpha}{\alpha}^{-1}\binom{n-k+\alpha}{\alpha}, \qquad k = 0,\ldots,n,
\end{equation}
and $c_k=0 $ otherwise. In the language of Hadamard products, writing $h_n^{\alpha}(z) = \sum_{k=0}^n c_k z^k$ and letting $T_{h_n^{\alpha}}$ act on the formal power series, we see that $\sigma_n^{\alpha}$ is nothing but the operator $T_{h_n^{\alpha}}$ that we introduced before. Hence, we may use  $\|\sigma_n^\alpha\|=\|T_{h_n^{\alpha}}\|$ to explore the behavior of Ces\`{a}ro means.

Let us pay more attention to two special cases.
\begin{enumerate}[(i)]
\item For $\alpha=0$ and $n\in\mathbb{N}$, the coefficients $c_k$ are  equal to $1$ for $k=0,\ldots, n$ and then they jump to zero for $k>n$. Hence, $\sigma_n^0$ is precisely equal to the $n$-th partial sum operator
\[
S_n f(z) = \sum_{k=0}^n a_k z^k,
\]
which is the Hadamard product of the Dirichlet kernel $D_n$ with $f$.  We know that $\|\sigma_n^0\|^2=n+1$ with the maximizing function $f(z)=nz^{n+1}-(n+1)z^n+1$.

\item For $\alpha=1$, we have that $\sigma_n^1=\sigma_n$ is the Ces\`{a}ro operator
\[
\sigma_n(f)= \sum_{k=0}^n\left(1-\frac{k}{n+1}\right)a_kz^k,
\]
which satisfies $\sigma_n(f)=K_n\ast f$, where $K_n$ is the classical Fej\'{e}r kernel. We have $\|\sigma_n^1\|^2=n/(n+1)$ with maximizing function $f(z)=z^{n+1}-(n+1)z+n$.
\end{enumerate}
See \cite{MNW1, MR4507242, MR4032210} for further detail.

\section{Main Results} \label{S:main-results}
Our main concern in this work is to estimate, as precisely as possible, the quantity $\|\sigma_n^\alpha\|$ for different values of the parameter $\alpha$. As mentioned above, the value $\alpha=1/2$ is a threshold point, and that is why in the following we have three theorems, with different flavors, about the behavior of $\|\sigma_n^\alpha\|$ corresponding to whether $\alpha>1/2$, $\alpha=1/2$, or $\alpha<1/2$. In the following, the notation $f(n)\sim g(n)$ means that
\[
\lim_{n\to\infty}\frac{f(n)}{g(n)}=1.
\]

\begin{theorem} \label{T:alpha<12}
Let $\alpha < \frac{1}{2}$. Then
\[
\|\sigma_n^\alpha\| \sim C_\alpha n^{\frac{1}{2}-\alpha},
\]
where
\[
C_\alpha := \Gamma(\alpha+1)\frac{\Gamma(1-2\alpha)^{1/2}}{\Gamma(1-\alpha)}
\]
is a finite positive constant.
\end{theorem}

\begin{theorem} \label{T:alpha=12}
Let $\alpha = \frac{1}{2}$. Then
\[
\|\sigma_n^\frac{1}{2}\| \sim \frac{1}{2}\log^{1/2} n.
\]
\end{theorem}

\begin{theorem} \label{T:alpha>12}
Let $\frac{1}{2} < \alpha < 1$. Then
\[
\max\left\{1,\frac{\alpha}{(2\alpha-1)^{1/2}}\frac{(2\alpha-1)^{\alpha-1/2}}{(2\alpha)^{\alpha}} \right\}
\leq
\liminf_{n \to \infty} \|\sigma_n^\alpha\|
\]
and 
\[
\limsup_{n \to \infty} \|\sigma_n^\alpha\|
\leq
\frac{\alpha}{(2\alpha-1)^{1/2}}.
\]
\end{theorem}

\section{Technical Lemmas} \label{S:estimation-Lemma}
In \cite{MR4032210}, it was shown that for a polynomial $h$ of degree $n$,
\begin{equation}\label{MasRanThm12-ii}
\left\|T_h\right\|^{2} \leq (n+1) \sum_{k=1}^{n}\left|c_{k+1}-c_{k}\right|^{2}.
\end{equation}
Additionally, this estimate was complemented in \cite{MR4179962} by the lower bounds
\begin{equation} \label{MPRThm22-ii}
\|T_h\|^2 \geq m \sum_{k=m}^{n}\left|c_{k+1}-c_{k}\right|^{2},
\end{equation}
which hold for every $m \in \{1,\ldots,n\}$. Using these results, we proceed to derive explicit asymptotic expressions for generalized Ces\`{a}ro means.
We begin with the following pair of inequalities, originally due to W. Gautschi \cite{gautschi1959}. For the reader's convenience, a sketch of the proof is provided.

\begin{lemma}[Gautschi's Inequality]
Let $x$ be a positive real number and $\alpha\in(0,1).$ Then
\begin{equation} \label{E:gautschi2}
(x+1)^{\alpha-1} < \frac{\Gamma(x+\alpha)}{\Gamma(x+1)} < x^{\alpha-1}.
\end{equation}
% In particular,
% \begin{equation} \label{E:gautschi3}
%     \frac{(x+1)^{\alpha-1}}{\Gamma(\alpha)}<\binom{x+\alpha-1}{\alpha-1}<\frac{x^{\alpha-1}}{\Gamma(\alpha)}.
% \end{equation}
\end{lemma}

\begin{proof}
Write $\Gamma(x+\alpha) = \Gamma\big((1-\alpha)x+\alpha(x+1)\big)$. Then, by the strict log-convexity of the Gamma function on the positive real axis \cite{artin1964}, we have
\[
\Gamma(x+\alpha) < \Gamma(x)^{1-\alpha}\Gamma(x+1)^\alpha
\]
for $x>0$ and $\alpha\in (0,1)$. Hence, by the central multiplication formula for the Gamma function, i.e., $\Gamma(x+1) = x\Gamma(x)$, $x>0$, we deduce that
\[
\Gamma(x+\alpha) < x^{\alpha-1}\Gamma(x+1),
\]
proving the second inequality in \eqref{E:gautschi2}.

For the first inequality, using similar arguments, we see that
\begin{align*}
\Gamma(x+1) &= \Gamma\big( \alpha(x+\alpha)+(1-\alpha)(x+\alpha+1)\big) \\
&<\Gamma(x+\alpha)^\alpha\Gamma(x+\alpha+1)^{1-\alpha} \\
&= \Gamma(x+\alpha) (x+\alpha)^{1-\alpha},
\end{align*}
concluding the proof.
\end{proof}

Notice that in the above proof, we actually obtained the lower bound $(x+\alpha)^{\alpha-1}$. However, since we are only concerned about the asymptotic behaviour as $x\to\infty$, we may equally use the lower quantity $(x+1)^{\alpha-1}$. We now present an optimal upper estimate for the quantity appearing in \eqref{MasRanThm12-ii}.

\begin{lemma} \label{L:estS}
Let $\alpha\in (0,1)$, $n>1$, and let $c_k$ be as in \eqref{E:defck}. Let
\[
S:= \sum_{k=1}^n |c_{k+1}-c_k|^2.
\]
Then, for $\alpha\neq \frac{1}{2}$,
\begin{equation} \label{E:estSneq0.5}
S \leq \Gamma(\alpha+1)^2 \frac{(n+1)^{2-2\alpha}}{(n+\alpha)^2 }\left(1+ \frac{(n-1)^{2\alpha-1}}{\Gamma(\alpha)^2(2\alpha-1)} +\frac{2\alpha-2}{\Gamma(\alpha)^2(2\alpha-1)} \right),
\end{equation}
and, for $\alpha = \frac{1}{2}$,
\begin{equation} \label{E:estS=0.5}
S \leq \frac{\pi}{4} \frac{n+1}{(n+\frac{1}{2})^2}\left(1+\frac{1}{\pi}( \log(n-1) +1) \right).
\end{equation}
\end{lemma}

\begin{proof}
Since $c_{n+1}=0$ and all the $c_k$'s are real, then
\[S = c_n^2 + \sum_{k=1}^{n-1} (c_{k+1}-c_k)^2.\]
For $k=1,\ldots,n-1$,
\begin{align*}
c_{k+1}-c_k &= \binom{n+\alpha}{\alpha}^{-1}\left(\binom{n-k-1+\alpha}{\alpha} - \binom{n-k+\alpha}{\alpha}\right) \\
&= \binom{n+\alpha}{\alpha}^{-1}\frac{1}{\Gamma(\alpha+1)}\left(\frac{\Gamma(n-k+\alpha)}{\Gamma(n-k)} - \frac{\Gamma(n-k+\alpha+1)}{\Gamma(n-k+1)}\right) \\
&= c_{k+1} \left(1 - \frac{n-k+\alpha}{n-k}\right) \\
&= -\frac{\alpha}{n-k} c_{k+1}.
\end{align*}
Hence,
\begin{align*}
S &= c_n^2\left( 1 + \frac{\alpha^2}{\Gamma(\alpha+1)^2}\sum_{k=1}^{n-1}\frac{1}{(n-k)^2} \frac{\Gamma(n-k+\alpha)^2}{\Gamma(n-k)^2}\right) \\
&= c_n^2\left( 1 + \frac{1}{\Gamma(\alpha)^2}\sum_{k=1}^{n-1} \frac{\Gamma(k+\alpha)^2}{\Gamma(k+1)^2}\right).
\end{align*}

For the leading factor $c_n$, we have
\begin{equation} \label{E:uppercn}
c_n^2 = \binom{n+\alpha}{\alpha}^{-2} =\frac{\Gamma(\alpha+1)^2 \Gamma(n+1)^2}{\Gamma(n+\alpha+1)^2} < \Gamma(\alpha+1)^2 \frac{(n+1)^{2-2\alpha}}{(n+\alpha)^2 },
\end{equation}
where we used Gautschi's inequality \eqref{E:gautschi2}. On the other hand, once more by \eqref{E:gautschi2},
\[
\sum_{k=1}^{n-1} \frac{\Gamma(k+\alpha)^2}{\Gamma(k+1)^2} < \sum_{k=1}^{n-1} k^{2\alpha-2}.
\]
Thus, estimating the sum with the corresponding integral, we see that
\[
\sum_{k=1}^{n-1} \frac{\Gamma(k+\alpha)^2}{\Gamma(k+1)^2} < 1 + \sum_{k=2}^{n-1} k^{2\alpha-2} \leq 1+ \int_1^{n-1} x^{2\alpha-2} \,dx.
\]

At this point, we have to distinguish between two cases. If $\alpha \neq \frac{1}{2}$, then
\[
\sum_{k=1}^{n-1} \frac{\Gamma(k+\alpha)^2}{\Gamma(k+1)^2} < \frac{(n-1)^{2\alpha-1}}{2\alpha-1} + \frac{2\alpha-2}{2\alpha-1},
\]
and therefore \eqref{E:estSneq0.5} is proved. But, if $\alpha=\frac{1}{2}$, we have
\[
\sum_{k=1}^{n-1} \frac{\Gamma(k+\alpha)^2}{\Gamma(k+1)^2} < \log(n-1) +1,
\]
which gives \eqref{E:estS=0.5}. In this case, we also need the well-known identities
\[
\Gamma\left(\frac{1}{2}\right)=\sqrt{\pi}, \qquad \Gamma\left(\frac{3}{2}\right)=\frac{\sqrt{\pi}}{2}.
\]
\end{proof}

At the same token, we provide a lower estimate for the quantity in \eqref{MPRThm22-ii}.

\begin{lemma}
Let $\alpha\in (0,1)$, $n>1$, $c_k$ be as in \eqref{E:defck}, and $m$ be a natural number with $1\leq m<n$. Let
\[
\tilde{S}_m := \sum_{k=m}^n |c_{k+1}-c_k|^2.
\]
Then, for $\alpha \neq \frac{1}{2}$,
\begin{equation} \label{E:Smneq0.5}
\tilde{S}_m > \Gamma(\alpha+1)^2 \frac{n^{2-2\alpha}}{(n+\alpha)^2 }\left(1+ \frac{(n-m+2)^{2\alpha-1}}{\Gamma(\alpha)^2(2\alpha-1)} - \frac{2^{2\alpha-1}}{\Gamma(\alpha)^2(2\alpha-1)}\right),
\end{equation}
and, for $\alpha = \frac{1}{2}$,
\begin{equation} \label{E:Sm=0.5}
\tilde{S}_m > \frac{\pi}{4} \frac{n}{(n+\frac{1}{2})^2}\left(1+\frac{1}{\pi}\big(\log(n-m+2)-\log2\big) \right).
\end{equation}
\end{lemma}

\begin{proof}
This proof has the same flavor as the previous one, but we are using the reverse inequalities in each case.
For $m<n$,
\begin{align*}
\tilde{S}_m &= c_n^2\left( 1 + \frac{1}{\Gamma(\alpha)^2}\sum_{k=m}^{n-1}\frac{\Gamma(n-k+\alpha)^2}{\Gamma(n-k+1)^2}\right)  \\
&= c_n^2\left( 1 + \frac{1}{\Gamma(\alpha)^2}\sum_{k=1}^{n-m}\frac{\Gamma(k+\alpha)^2}{\Gamma(k+1)^2}\right) .
\end{align*}
For the coefficient $c_n$, we have
\begin{equation} \label{E:lowercn}
c_n^2 =\frac{\Gamma(\alpha+1)^2 \Gamma(n+1)^2}{\Gamma(n+\alpha+1)^2} > \Gamma(\alpha+1)^2 \frac{n^{2-2\alpha}}{(n+\alpha)^2}.
\end{equation}
Note that in \eqref{MPRThm22-ii} we were allowed to pick $m=n$. However, in this case $\tilde{S}_n = c_n^2$, and the required estimate is precisely the established inequality \eqref{E:lowercn}. Then
\[
\sum_{k=1}^{n-m} \frac{\Gamma(k+\alpha)^2}{\Gamma(k+1)^2} > \sum_{k=1}^{n-m} (k+1)^{2\alpha-2} \geq \int_1^{n-m+1} (x+1)^{2\alpha-2} \,dx,
\]
and we conclude the proof by evaluating the integrals in the cases $\alpha \neq \frac{1}{2}$ and $\alpha = \frac{1}{2}$.
\end{proof}

Note that comparing \eqref{E:uppercn} and \eqref{E:lowercn}, we obtain the asymptotic
\begin{equation} \label{E:asympcn}
c_n^2 \sim \frac{\Gamma(\alpha+1)^2}{n^{2\alpha}}.
\end{equation}

\section{Proofs of the main results} \label{S:proofs}

\subsection{Proof of Theorem \ref{T:alpha<12}}
By Gautschi's inequality \eqref{E:gautschi2},
\begin{equation}\label{E:stirling-2}
\frac{\Gamma(k+\alpha)^2}{\Gamma(k+1)^2} < k^{2(\alpha-1)}, \qquad k\geq 1.
\end{equation}
Since in this case $0<\alpha < \frac{1}{2}$, we conclude that
\[
\sum_{k=1}^{\infty} \frac{\Gamma(k+\alpha)^2}{\Gamma(k+1)^2} < \infty.
\]
In fact, we can go further and observe that
\[
\frac{1}{\Gamma(\alpha)^2}\sum_{k=0}^{\infty} \frac{\Gamma(k+\alpha)^2}{\Gamma(k+1)^2}=
\sum_{k=0}^\infty \binom{k+\alpha-1}{\alpha-1}^2 = \sum_{k=0}^\infty \binom{k+\alpha-1}{k}^2
\]
is the $H^2$-norm of the function $\sum_{k=0}^\infty \binom{k+\alpha-1}{k}z^k=(1-z)^{-\alpha}$ ($|z|<1$), so by Parseval's theorem we have 
\[
\frac{1}{\Gamma(\alpha)^2}\sum_{k=0}^{\infty} \frac{\Gamma(k+\alpha)^2}{\Gamma(k+1)^2} = \frac{1}{2\pi}\int_{-\pi}^\pi |1-e^{i\theta}|^{-2\alpha}\,d\theta.
\]
This can be rewritten as the double integral
\[
\frac{1}{2\pi}\int_{-\pi}^\pi |1-e^{i\theta}|^{-2\alpha}\,d\theta=\frac{1}{(2\pi)^2}\int_{-\pi}^\pi\int_{-\pi}^\pi |e^{i\theta_1}-e^{i\theta_2}|^{-2\alpha}\,d\theta_1d\theta_2,
\]
by rotational invariance of the Lebesgue measure. Hence, we obtain the Morris integral for $k=2$ (or a version of the Selberg integral \cite{MR18287, MR3287209, MR3308963}), which  equals 
\[
\frac{1}{(2\pi)^2}\int_{-\pi}^\pi\int_{-\pi}^\pi |e^{i\theta_1}-e^{i\theta_2}|^{-2\alpha}\,d\theta_1d\theta_2 = \frac{\Gamma(1-2\alpha)}{\Gamma(1-\alpha)^2}.
\]

Now, on the one hand, by \eqref{MasRanThm12-ii},
\[
\|\sigma_n^\alpha\|^2 \leq (n+1)\sum_{k=1}^n |c_{k+1}-c_k|^2 = (n+1)c_n^2\left( 1 + \frac{1}{\Gamma(\alpha)^2}\sum_{k=1}^{n-1} \frac{\Gamma(k+\alpha)^2}{\Gamma(k+1)^2}\right).
\]
Then, by the asymptotic \eqref{E:asympcn}, we have
\begin{eqnarray*}
\limsup_n \frac{\|\sigma_n^\alpha\|^2}{n^{1-2\alpha}} &\leq& \Gamma(\alpha+1)^2 \left(1+\frac{1}{\Gamma(\alpha)^2}\sum_{k=1}^{\infty} \frac{\Gamma(k+\alpha)^2}{\Gamma(k+1)^2}\right) \\
&=& \Gamma(\alpha+1)^2 \frac{\Gamma(1-2\alpha)}{\Gamma(1-\alpha)^2} .
\end{eqnarray*}

On the other hand, put
\[
m:=\left[\frac{n-1}{2^\gamma}\right],
\]
where $[\cdot]$ denotes the integer part and $\gamma\in(0,1)$. Then, by \eqref{MPRThm22-ii}, we obtain
\begin{align*}
\|\sigma_n^\alpha\|^2 \geq m \tilde{S}_m = m c_n^2\left( 1 + \frac{1}{\Gamma(\alpha)^2}\sum_{k=1}^{n-m}\frac{\Gamma(k+\alpha)^2}{\Gamma(k+1)^2}\right).
\end{align*}
Notice that $m \geq \frac{n-1}{2^\gamma}-1$ and
\[
n-m \geq n-\frac{n-1}{2^\gamma}=\frac{(2^\gamma-1)n+1}{2^\gamma},
\]
so that $n-m\to+\infty$ as $n\to+\infty$, and for every $\gamma\in(0,1)$,
\[
\liminf_n \frac{\|\sigma_n^\alpha\|^2}{n^{1-2\alpha}}\geq \frac{\Gamma(\alpha+1)^2}{2^\gamma}\frac{\Gamma(1-2\alpha)}{\Gamma(1-\alpha)^2}.
\]
Taking the limit as $\gamma\to 0^+$, we conclude the proof. \qed

\subsection{Proof of Theorem \ref{T:alpha=12}}
 By \eqref{MasRanThm12-ii} and \eqref{E:estS=0.5}, we have
\[
\frac{\|\sigma_n^\frac{1}{2}\|^2}{\log n}\leq \frac{\pi}{4} \frac{(n+1)^2}{(n+\frac{1}{2})^2}\left(\frac{1}{\log n}+\frac{\log(n-1)}{\pi \log n}+ \frac{1}{\pi\log n} \right),
\]
revealing that
\[
\limsup_{n} \frac{\|\sigma_n^\frac{1}{2}\|^2}{\log n}\leq \frac{1}{4}.
\]

Now, let $\gamma \in(0,1)$, and set
\[
m:=\left[\frac{n-1}{2^\gamma}\right].
\]
Then we have
\[
m\leq \frac{n-1}{2^\gamma} < n-1
\]
and that, at least for $n\geq 7$,
\[
m\geq \frac{n-1}{2^\gamma}-1 \geq \frac{n-1}{2}-1 = \frac{n-3}{4} \geq 1.
\]
Therefore, for every $n\geq 7$, equations \eqref{MPRThm22-ii} and \eqref{E:Sm=0.5} yield
\begin{align*}
\|\sigma_n^\frac{1}{2}\|^2 &> \frac{\pi}{4} \frac{mn}{(n+\frac{1}{2})^2}\left(1+\frac{1}{\pi}\big(\log(n-m+2)-\log2\big) \right) \\
&\geq \frac{\pi}{4} \left(\frac{n-1}{2^\gamma}-1\right)\frac{n}{(n+\frac{1}{2})^2}\left(1+\frac{1}{\pi}\left(\log\left(\frac{n(2^\gamma-1)+1+2^{\gamma+1}}{2^\gamma}\right)-\log2\right) \right).
\end{align*}
In particular, for every $\gamma \in (0,1)$,
\[
\liminf_n \frac{\|\sigma_n^\frac{1}{2}\|^2}{\log n}\geq \frac{1}{2^\gamma}\frac{1}{4},
\]
and the theorem follows by taking the limit as $\gamma\to 0^+$. \qed

\subsection{Proof of Theorem \ref{T:alpha>12}}

By \eqref{MasRanThm12-ii} and \eqref{E:estSneq0.5},
\begin{align*}
\|\sigma_n^\alpha\|^2 &\leq (n+1)\sum_{k=1}^n |c_{k+1}-c_k|^2\\
&\leq \Gamma(\alpha+1)^2 \frac{(n+1)^{3-2\alpha}}{(n+\alpha)^2 }\left(1+ \frac{(n-1)^{2\alpha-1}}{\Gamma(\alpha)^2(2\alpha-1)} +\frac{2\alpha-2}{\Gamma(\alpha)^2(2\alpha-1)} \right).
\end{align*}
In particular, since $\frac{1}{2} < \alpha <1$,
\[
\limsup_{n} \|\sigma_n^\alpha\|^2 \leq \frac{\alpha^2}{2\alpha-1}.
\]

For the other estimate, by Lemma \ref{P:eigenvalues},
\[
\|\sigma_n^\alpha\|^2 \geq |c_1|^2 = \frac{\Gamma(n+1)^2}{\Gamma(n+\alpha+1)^2}\frac{\Gamma(n+\alpha)^2}{\Gamma(n)^2}=\frac{n^2}{(n+\alpha)^2},
\]
and thus $\liminf_n\|\sigma_n^\alpha\|^2 \geq 1$. Moreover, by \eqref{MPRThm22-ii} and \eqref{E:Smneq0.5}, we have
\begin{align*}
\|\sigma_n^\alpha\|^2 \geq m\tilde{S}_m > m\Gamma(\alpha+1)^2 \frac{n^{2-2\alpha}}{(n+\alpha)^2 }\left(1+ \frac{(n-m+2)^{2\alpha-1}}{\Gamma(\alpha)^2(2\alpha-1)} - \frac{2^{2\alpha-1}}{\Gamma(\alpha)^2(2\alpha-1)}\right),
\end{align*}
for all $m\in\{1,\ldots,n\}$. Put $m=\left[\frac{n-1}{2\alpha}\right]$. This particular choice yields
\[
\|\sigma_n^\alpha\|^2 \geq \frac{n-1-2\alpha}{2\alpha}\Gamma(\alpha+1)^2 \frac{n^{2-2\alpha}}{(n+\alpha)^2 }\left(1+ \frac{\left(\frac{(2\alpha-1)n+3}{2\alpha}\right)^{2\alpha-1}}{\Gamma(\alpha)^2(2\alpha-1)} - \frac{2^{2\alpha-1}}{\Gamma(\alpha)^2(2\alpha-1)}\right),
\]
so that
\[
\liminf_n\|\sigma_n^\alpha\|^2 \geq \frac{1}{2\alpha}\Gamma(\alpha+1)^2\Big(\frac{2\alpha-1}{2\alpha}\Big)^{2\alpha-1}\frac{1}{\Gamma(\alpha)^2(2\alpha-1)} = \frac{\alpha^2(2\alpha-1)^{2\alpha-2}}{(2\alpha)^{2\alpha}}.
\]
\qed

\section{Final Remarks}
We conclude this work with the following comments.
\begin{enumerate}[(i)]
\item Theorems \ref{T:alpha<12}, \ref{T:alpha=12} and \ref{T:alpha>12} reaffirm that $\sigma_n^\alpha(f)\to f$ in $\mathcal{D}_\omega$ if and only if $\alpha > 1/2$, as it was established in \cite{MR4179962}. Additionally, the behavior at the critical case $\alpha=1/2$ was conjectured and observed numerically in the doctoral dissertation of P.-O. Paris\'{e}. See \cite{MR4179962}.

\item For $\alpha<\frac{1}{2}$, by Theorem \ref{T:alpha<12} we have that $\|\sigma_n^\alpha\|\sim C_\alpha n^{\frac{1}{2}-\alpha}$, with
\[C_\alpha = \Gamma(\alpha+1)\frac{\Gamma(1-2\alpha)^{1/2}}{\Gamma(1-\alpha)}.\]
For each fixed $n\geq1$, interpreting $ C_\alpha n^{\frac{1}{2}-\alpha}$ as a continuous function of $\alpha$, we have that
$$\lim_{\alpha \to 0^+} C_\alpha n^{\frac{1}{2}-\alpha} = n^{\frac{1}{2}},$$
which is coherent with the fact that the norm of $S_n=\sigma_n^0$ asymptotically behaves as $\sqrt{n}$.
Moreover, it was showed in the proof of Theorem \ref{T:alpha<12} that
\[C_\alpha=\Gamma(\alpha+1)\left(1+\frac{1}{\Gamma(\alpha)^2}\sum_{k=1}^{\infty} \frac{\Gamma(k+\alpha)^2}{\Gamma(k+1)^2}\right)^\frac{1}{2},\]
so that by Fatou's Lemma
\[
\liminf_{\alpha\to\frac{1}{2}^-} C_\alpha n^{\frac{1}{2}-\alpha} \geq \frac{1}{2}\left(\sum_{k=1}^{\infty}\frac{\Gamma(k+\frac{1}{2})^2}{\Gamma(k+1)^2} \right)^{\frac{1}{2}} = +\infty,
\]
and it heuristically matches the established result for $\alpha=\frac{1}{2}$,
\[
\frac{1}{2}\left(\sum_{k=1}^{n}\frac{\Gamma(k+\frac{1}{2})^2}{\Gamma(k+1)^2} \right)^{\frac{1}{2}} \sim \frac{1}{n}\left(\sum_{k=1}^{n}\frac{1}{k}\right)^\frac{1}{2} \sim \frac{1}{2}\log^\frac{1}{2} n \sim \|\sigma_n^{\frac{1}{2}}\|.
\]
Based on the estimations in the proof of Theorem \ref{T:alpha>12}, similar remarks can be made for $\alpha \to \frac{1}{2}^+$ and also for $\alpha \to 1$, which are consistent with the findings in \cite{MNW1}.

\item In Theorems \ref{T:alpha=12} and \ref{T:alpha<12}, we provided precise asymptotic values for $\|\sigma_n^\alpha\|$ whenever $\alpha=1/2$ and $\alpha<1/2$, respectively. In \cite{MR4507242}, the case $\alpha =0$ was studied in detail, and precise formulas for $\|\sigma_n^0\|$ were obtained in three different but equivalent norms. Therefore, the next step in this direction is to derive precise formulas for $\|\sigma_n^\alpha\|$, for a fixed $0<\alpha \leq 1/2$.

\item Handling the case $\alpha > 1/2$ presented some challenges. For other values of $\alpha$, the norm $\|\sigma_n^\alpha\|$ diverges as $n \to \infty$, allowing us to distinguish and isolate a dominant term which leads to an asymptotic formula for the growth at infinity. This feature enables us to disregard other terms that were overshadowed by the principal one. However, whenever $\alpha > 1/2$, the norm $\|\sigma_n^\alpha\|$ remains bounded as $n$ grows. Consequently, all contributing terms become important, making it more complex to determine the precise, or even approximate, value of $\|\sigma_n^\alpha\|$ as $n$ increases. The only case for which precise formulas have been derived is $\alpha=1$ \cite{MNW1}. Therefore, building upon insights from \cite{MNW1}, what is the exact formula for $\|\sigma_n^\alpha\|$, for a fixed $\alpha > 1/2$?
\end{enumerate}

\section*{Acknowledgments}
Parts of the work that contributed to this paper were conducted during the first author's visit to Universit\'{e} Laval, and he expresses his gratitude for their hospitality.
\\

\noindent \textbf{Competing interests}: The authors declare none.

\bibliographystyle{acm}
\bibliography{Cesaro-Operator-on-Dzeta-References}

\end{document}